\documentclass[a4paper]{article}
\usepackage{amsmath,amssymb,amsfonts,amsthm,mathrsfs}
\usepackage{color,bm}

\newtheorem{definition}{Definition}[section]

\newtheorem{teo}{Theorem}[section]

\def \pmatrix{ \left( \begin{array} }
\def \endpmatrix{ \end{array} \right) }
\def\Cl{{\mathcal{C} \ell}_{0,n}}
\def\R{\mathbb{R}}
\def\C{\mathbb{C}}
\def\H{\mathbb{H}}
\def\N{\mathbb{N}}

\def\A{\mathcal{A}}
\def\xv{\underline{x}}
\def\P{\mathcal{P}}
\def\Pk{\P_{k}}
\def\Pkn{\Pk^n}
\def\xv{\underline{x}}

\begin{document}

\begin{center}%
\vskip 1.5em%
{\Large On generalized Vietoris' number sequences - origins, properties and applications}%
\vskip 1.5em%
{\large
\lineskip .5em%
\begin{tabular}[t]{ccc}%
I. Ca\c{c}\~{a}o$^a$ & M.~I.~Falc\~{a}o$^b$ & H. R. Malonek$^a$\\
{\small	isabel.cacao@ua.pt}&\small	mif@math.uminho.pt& \small hrmalon@ua.pt
\end{tabular}\par}%

\vskip 0.5cm
{\small 
\vskip 0.5cm

$^a$ CIDMA  and Department of Mathematics,	University of Aveiro,  Aveiro, Portugal
	
$^b$ CMAT and Department of Mathematics and Applications, University of Minho,\break Campus de Gualtar, Braga, Portugal}
\end{center}%
\vskip 0.5cm
\begin{abstract}
Ruscheweyh and Salinas showed in 2004 the relationship of a celebrated theorem of Vietoris (1958) about the positivity of certain sine and cosine sums with the function theoretic concept of stable holomorphic functions in the unit disc. The present paper shows that the coefficient sequence in Vietoris' theorem is identical with the number sequence that characterizes generalized Appell sequences of homogeneous Clifford holomorphic polynomials in $\R^3.$ The paper studies one-parameter generalizations of Vietoris' number sequence, their properties as well as their role in the framework of Hypercomplex Function Theory.
\end{abstract}

identities; monogenic Appell polynomials

\section{Introduction}
\subsection{Vietoris' theorem and its coefficient's sequence }

In \cite{viet58} Vietoris proved the positivity of two trigonometric sums with a pairwise coefficient's sequence defined by coefficients with even index of the form
\begin{equation}\label{ck viet}
a_{2m} = \frac{1}{2^{2m}}\binom{2m}{m}\;\; m \geq 0.
\end{equation}
We refer to Vietoris' theorem as it is given by Askey and Steinig in \cite{asst74}:

\begin{teo}[L. Vietoris]\label{TheoViet}
If $a_{2m}$ ($m \geq 0$) are given by (\ref{ck viet}) and the following adjacent coefficients are supposed to be the same, i.e. $a_{2m+1}:=a_{2m}$ then
\begin{eqnarray}\label{TheoViet2}
\sigma_n(x) \hspace{-0.2cm}&=&\hspace{-0.2cm} \sum_{k=1}^{n} a_k\sin kx > 0,\;\;\; 0 < x < \pi, \\
\tau_n(x) \hspace{-0.2cm}&=&\hspace{-0.2cm} \sum_{k=0}^{n} a_k\cos kx > 0,\;\;\; 0 < x < \pi. \label{TheoViet3}
\end{eqnarray}
\end{teo}
Those coefficients are central binomial coefficients with particular weight $\frac{1}{2^{2m}}$. Their similarity to Catalan numbers $\mathcal{C}_m$ defined by the central binomial coefficients but with different weight by $\mathcal{C}_m = \frac{1}{{m+1}}\binom{2m}{m}$ is obvious. Like for Catalan's numbers, the appearance of (\ref{ck viet}) in a vast number of combinatorial identities, Special Functions or integral representations reveals some particular role of this number sequence.

Using the arsenal of real analysis methods in positivity theory the authors of \cite{asst74} showed the embedding of Vietoris' results in general problems for Jacobi polynomials and extensions of inequalities considered by Fej\'{e}r\footnote{The statement of Theorem 1.1. is also true if $a_0\geq a_1\geq \dots\geq a_n > 0$ and $(2k)a_{2k}\leq (2k-1)a_{2k-1},\;k\geq 1,$ but Theorem 1.1., being a special case of this second variant, implies also this second variant, which means that both variants are equivalent, cf. \cite{asst74}.}. Moreover, it is well known that in real analysis trigonometric sums are relevant in Fourier analysis in general.  Number theory and the theory of univalent functions are examples of other fields where they play an important role. For more details from the real analysis point of view we refer to Askey's \cite{ask94}. On page 5 of this book the coefficients $a_k$ are introduced in the form $a_{2k}= a_{2k+1}= (\frac{1}{2})_k/k!, k\geq 0, $ where $(\cdot)_k$ is the raising factorial in the classical form of the Pochhammer symbol (see subsection 2.2).
Thirty years after the publication of \cite{asst74}, Ruscheweyh and Salinas showed in \cite{rusa04} the relevance of the celebrated result of Vietoris for a \textit{complex function theoretic result} on the stability of the function $f(z)=\sqrt{(1+z)/(1-z)}$ in the context of subordination of analytic functions in the unit disc $\mathbb{D}.$

The intention of this paper is to reveal some facts about the role of the sequence (\ref{ck viet}) and its generalization in the context of \textit{Hypercomplex Function Theory} (HFT) i.e. the theory of functions in higher dimensional Euclidean spaces which uses instruments of non-commutative Clifford algebras, cf. \cite{desosu92}.

For this purpose we consider the following sequence $(c_k)_{k\ge 0}$ with
\begin{equation}\label{equ1a}
c_k = \frac{1}{2^k}\binom{k}{\lfloor\tfrac{k}{2}\rfloor}, \;\; k\geq 0.
\end{equation}
By simple calculation we obtain directly from (\ref{equ1a}) that
\begin{equation*}
    c_{2m-1} = \frac{1}{2^{2m-1}}\binom{2m-1}{m-1} = \frac{1}{2^{2m-1}}\frac{(2m-1)!2m}{m!(m-1)! 2m} = \frac{1}{2^{2m}}\binom{2m}{m}=c_{2m}, \;\; m \geq 0.
\end{equation*}
It is evident that this sequence is exactly the sequence of the coefficients $a_k$ in Vietoris' \textit{sine sum} (\ref{TheoViet2}) in \textit{closed form}.
The difference to the \textit{cosine sum} coefficients in (\ref{TheoViet3}), starting from $k=0,$ consist only in the inclusion of $a_0=1$ at the beginning and the corresponding shifts of the indices. This is the reason why we call (\ref{equ1a}) the \textit{Vietoris' number sequence.} Its first elements are
\begin{equation}\label{Vietsequ}
    c_0=1;\;c_1=c_2=\frac{1}{2};\;c_3=c_4=\frac{3}{8};\;c_5=c_6=\frac{5}{16}; \dots
\end{equation}
As we will see, Vietoris' sequence plays a remarkable role in the theory of generalized sequences of Appell polynomials in HFT where generalized holomorphic functions defined in $\R^{n+1}$ with values in a non-commutative Clifford algebra $\Cl$ are considered. Whereas Vietoris' sequence corresponds to the special hypercomplex case $n=2$ we will show that also a natural generalization of (\ref{equ1a}) to the case $\R^{n+1}, n > 2,$ exists.

\subsection{Hamilton's quaternions come into the play}
 The main intention of this paper is to show how generalized Vietoris' number sequences appear naturally in HFT. Far from being exhaustive in our explanation, it seemed to be worthwhile for us to show the value of HFT even on the the level of certain number relations.

Anticipating the result of the general case, we would like to illustrate the surprising appearance of (\ref{Vietsequ}) in a relation between the generators of a non-commutative Clifford algebra $\Cl$ for $n=2.$ Therefore we use Hamilton's well known non-commutative algebra $\H$ of quaternions (cf. \cite{guehasp08})
\begin{equation*}
 q = x_0 + x_1 \mathbf{i} +x_2 \mathbf{j} + x_3\mathbf{k},\;\;\mbox{where}\;\;\mathbf{i}^2 =\mathbf{j}^2 = \mathbf{k}^2=\mathbf{i}\mathbf{j}\mathbf{k}= -1.
\end{equation*}
Due to non-commutativity the formal expansion of the binomial
$(\mathbf{i}+\mathbf{j})^k,\;k \geq 0,$
will not directly lead to Pascal's triangle, as the case $k=3$ shows:
\begin{equation}\label{cubicexp}
  (\mathbf{i}+\mathbf{j})^3= \mathbf{i}^3+ (\mathbf{i}\mathbf{i} \mathbf{j}+ \mathbf{i}\mathbf{j}\mathbf{i} + \mathbf{j}\mathbf{i}\mathbf{i}) + (\mathbf{i}\mathbf{j} \mathbf{j}+ \mathbf{j}\mathbf{i}\mathbf{j} + \mathbf{j}\mathbf{j}\mathbf{i}) + \mathbf{j}^3
\end{equation}
But that happens if we try to embed the non-commutative multiplication into the concept of a $k-nary$ symmetric (or permutative) operation. Therefore let $a_i$ stay for one of the generators $\mathbf{i}$ or $\mathbf{j}$ and write the quaternionic $k$-fold product of $k-s$ generators $\mathbf{i}$ and $s$ generators $\mathbf{j},$ respectively, in the general form of a symmetric $``\times"$ product (\cite{ma90b}), i.e.
\begin{equation}\label{symm conven}
    \mathbf{i}^{k-s}\times\mathbf{j}^{s}:=\frac{1}{k!}\sum_{\pi(i_1, \ldots ,i_n)}{} a_{i_1}a_{i_2} \cdots a_{i_k}
\end{equation}
where the sum runs over {\bf {\em all}} permutations of all
$(i_1, \ldots ,i_n)$.
Then, by taking into account the repeated use of $\mathbf{i}$ and $\mathbf{j}$ on the right hand side of (\ref{symm conven}), we can write
\begin{equation}\label{symm conven1}
    \mathbf{i}^{k-s}\times\mathbf{j}^{s}= \frac{(k-s)!s!}{ k!}\hspace{-0.2cm}\sum_{\pi(i_1, \ldots ,i_n)}{} a_{i_1}a_{i_2} \cdots a_{i_k}= \left[\binom{k}{s}\right]^{-1}\hspace{-0.2cm}\sum_{\pi(i_1, \ldots ,i_n)}{} a_{i_1}a_{i_2} \cdots a_{i_k}
\end{equation}
where now the sum runs only over {\bf {\em all distinguished}} permutations of all
$(i_1, \ldots ,i_n)$.
Applying, for example, the convention (\ref{symm conven}) to (\ref{cubicexp}) we obtain now for $k=3$ the expansion written with binomial coefficients in the form
\begin{equation*}\label{cubicexp_rev}
  (\mathbf{i}+\mathbf{j})^3= \binom{3}{0}\mathbf{i}^3+ \binom{3}{1}\mathbf{i}^{2}\times\mathbf{j} + \binom{3}{2} \mathbf{i}\times\mathbf{j}^{2} + \binom{3}{3}\mathbf{j}^3.
\end{equation*}
Analogously, the expansion of $(\mathbf{i}+\mathbf{j})^k$ for any $k\geq 0$ follows the rules of the ordinary binomial expansion in an obvious way\footnote{An obvious generalization of (\ref{symm conven}) to the case of more than two generators used in the general case of $\Cl$ for $n\geq2$ leads to a polynomial formula and will be introduced in Section 2. }.

Let $(A_k)_{k\geq 0}$ be the sequence defined by
\begin{equation}\label{pattern}
    A_k=(-1)^k\left[\sum_{s=0}^{k}\binom{k}{s}(\mathbf{i}^{k-s} \times\mathbf{j}^{s})^2 \right]^{-1}.
\end{equation}
Its first elements are
\begin{align*}
A_0=&1\\
A_1=&(-1)^1\left[\mathbf{i}^2+ \mathbf{j}^2\right]^{-1}=\frac{1}{2}\\
A_2=&(-1)^2\left[(\mathbf{i}^2)^2+ \binom{2}{1}(\mathbf{i} \times \mathbf{j})^2 + (\mathbf{j}^2)^2\right]^{-1}=\frac{1}{2}\\
A_3=&(-1)^3\left[(\mathbf{i}^3)^2+ \binom{3}{1}(\mathbf{i}^2 \times\mathbf{j})^2 + \binom{3}{2}(\mathbf{i} \times \mathbf{j}^2)^2 +  (\mathbf{j}^3)^2 \right]^{-1}=\frac{3}{8}\\
A_4=&(-1)^4\left[(\mathbf{i}^4)^2+ \binom{4}{1}(\mathbf{i}^3 \times\mathbf{j})^2 + \binom{4}{2}(\mathbf{i}^2 \times \mathbf{j}^2)^2 + \binom{4}{3}(\mathbf{i} \times \mathbf{j}^3)^2 +(\mathbf{j}^4)^2 \right]^{-1}=\frac{3}{8}.
\end{align*}
Notice that for $k\geq2$ the influence of the non-commutativity is evident since, for example, $\mathbf{i} \times \mathbf{j}=\frac{1!1!}{2!}(\mathbf{i}\mathbf{j} + \mathbf{i}\mathbf{j})=0$ or
$\mathbf{i}^{2}\times\mathbf{j}=\binom{3}{1}^{-1}(\mathbf{i}\mathbf{i} \mathbf{j}+ \mathbf{i}\mathbf{j}\mathbf{i} + \mathbf{j}\mathbf{i}\mathbf{i})=-\frac{1}{3}\mathbf{j}$ and $\mathbf{i}\times\mathbf{j}^2=\binom{3}{1}^{-1}(\mathbf{i}\mathbf{j} \mathbf{j}+ \mathbf{j}\mathbf{i}\mathbf{j} + \mathbf{j}\mathbf{j}\mathbf{i})=-\frac{1}{3}\mathbf{i}.$
As the example for the first values of $k$ shows, the  $A_k$ obtained by formula (\ref{pattern}) coincide exactly, including their pairwise appearance, with those of Vietoris' sequence in (\ref{Vietsequ}). The proof of this fact as special case of the formula for generalized Vietoris's number sequences obtained by HFT methods will be given in Section 4.

\section{Basic concepts of Hypercomplex Function Theory}
\subsection{Clifford holomorphic functions and multidimensional polynomial sequences}
Obviously, it cannot be our aim to include here all proofs of facts from HFT that we will mention. What concerns basic facts form HFT, we refer e. g. to \cite{desosu92,guehasp08}, for approaches to multidimensional polynomial sequences by methods of HFT see e. g. \cite{bogue10,ma90b}.

 Let $\{e_1,e_2,\ldots,e_n\}$ be an orthonormal basis of the Euclidean vector space $\R^{n}$ with a non-commutative product according to the multiplication rules
$$e_ke_l+e_le_k=-2\delta_{kl},\quad k,l=1,\ldots,n,$$
where $\delta_{kl}$ is the Kronecker symbol. The set $\{e_A:A\subseteq \{ 1,\ldots,n\}\}$ with
$e_A=e_{h_1}e_{h_2}\cdots e_{h_r}$, $ 1\leq h_1 < \cdots < h_r\le n,\, e_{\emptyset}=e_0=1,$
is a basis of the $2^n$-dimensional Clifford algebra $\Cl$ over $\mathbb{R}.$ Let $\mathbb{R}^{n+1}$ be embedded in $\Cl$ by identifying $(x_0,x_1,\ldots,x_n)\in\mathbb{R}^{n+1}$ with
\begin{gather*}
x=x_0+\xv\in\mathcal{A}_n:=\hbox{span}_{\R}\{1,e_1,\ldots,e_n\}\subset\Cl.
\end{gather*}
Here, $x_0=Sc(x)$ and $\xv= V(x)=x_1e_1+\cdots+x_ne_n$ are the scalar and vector parts of the paravector $x \in\mathcal{A}_n.$ The conjugate of $x$ is given by $\bar{x}=x_0-\xv$
and its norm by $|x|=(x\bar{x})^{\frac{1}{2}}=(x_0^2+x_1^2+\cdots+x_n^2)^{\frac{1}{2}}.$
Of course, $\Cl$ for $n=1$ is the algebra of complex numbers $\C$ and the case $n=2,$ already mentioned in the previous section, corresponds to $\H,$ the algebra of Hamilton's quaternions generated by $e_1=\mathbf{i}, e_2=\mathbf{j}$ with $\mathbf{k}:=e_1e_2.$

Needless to mention that $\Cl$-valued functions defined in some open subset $\Omega\subset \mathbb{R}^{n+1}$ are of the form $f(z)=\sum_A f_A(x)e_A$ with real valued $f_A(x).$

The relationship of HFT to complex function theory can also easily be illustrated by the generalization of the complex Wirtinger derivatives in form of generalized Cauchy-Riemann operators in $\mathbb{R}^{n+1},\ n\ge 1,$
\begin{equation*}
\overline{\partial}:=\frac{1}{2}(\partial_0+\partial_{\xv})\;\;\mbox{and its conjugate}\;\;{\partial}:=\frac{1}{2}(\partial_0-\partial_{\xv})
\end{equation*}
with
\begin{equation*}
\partial_0:=\frac{\partial}{\partial x_0} \quad \mbox{and}\quad \partial_{\xv}:=e_1\frac{\partial}{\partial x_1}+\cdots+e_n\frac{\partial}{\partial x_n}.
\end{equation*}

$\mathcal{C}^1$-functions $f$ in the kernel of $\overline{\partial},$ i.e. with $\overline{\partial}f=0$ (resp. $f\overline{\partial}=0$) are called {\it left Clifford holomorphic} (resp. {\it right Clifford holomorphic}), cf. \cite{guehasp08}, or \textit{left resp. right monogenic} \cite{desosu92}. We suppose that $f$ is hypercomplex-differentiable in $\Omega$ in the sense of \cite{guema99}, that is, it has a uniquely defined areolar derivative $f'$ in each point of $\Omega$. Then, $f$ is real-differentiable and $f'$ can be expressed by the conjugate generalized Cauchy-Riemann operator as $f' =\partial f$. Since a hypercomplex differentiable function belongs also to the kernel of $\overline{\partial}$, one has $f'=\partial_0 f= - \partial_{\xv}f$ like in the complex case.

The basis for a suitable analog to power series in HFT by using alternatively several hypercomplex variables has been developed in \cite{ma90b}. The starting point is a second hypercomplex structure of $\R^{n+1}$ different from that given by $\mathcal{A}_n$ which relies on the isomorphism
$$
\R^{n+1}\cong{{\cal{H}}^n}=\{\vec{z}: z_k=x_k -x_0e_{k}; x_0,x_k\in \R, \;k=1,\ldots,n\},
$$
so that $\Cl $-valued functions are considered as mappings
\begin{eqnarray*}
f: \Omega \subset \R^{n+1}\cong{{\cal{H}}^n} &\longmapsto & {\mathcal C\ell}_{0,n}.
\end{eqnarray*}
Furthermore, the general version of the $``\times"$-product (\ref{symm conven}) used in the introduction, is defined by
\begin{definition}
Let $V_{+,\cdot}$ be a commutative or non-commutative ring, $a_k \in V\;
(k=1, \ldots ,n)$. The $``\times"$-product is defined by
\begin{equation}\label{crossprod}
a_1 \times a_2 \times \cdots \times a_n = \frac{1}{n!}
    \sum_{\pi(i_1, \ldots ,i_n)}{} a_{i_1}a_{i_2} \cdots a_{i_n}
\end{equation}
where the sum runs over {\bf {\em all}} permutations of all
$(i_1, \ldots ,i_n)$.
\end{definition}
Moreover, if the factor $a_j$ occurs ${\mu}_j$-times in
(\ref{crossprod}), we briefly write
\begin{equation*}\label{conv geral}
a_1 \times a_2 \times \cdots \times a_n = {a_1}^{{\mu}_1} \times {a_2}^{{\mu}_2} \times \cdots \times {a_n}^{{\mu}_n}
\end{equation*}
and set parentheses if the powers are understood in the ordinary way.
In analogy to (\ref{symm conven1}) the fact that
\begin{eqnarray*}
{a_1}^{\mu_1} \times {a_2}^{\mu_2} \times \cdots \times {a_n}^{\mu_n}
 = \frac{{\mu}!}{|{\mu}|!}
    \sum_{\pi(i_1, \ldots ,i_{|\mu|})}{} a_{i_1}a_{i_2} \cdots a_{i_{|{\mu}|}}
\end{eqnarray*}
where the sum runs over {\bf {\em all distinguished}} permutations,
leads to the following polynomial formula in terms of a multi-index $\mu=(\mu_1, \ldots \mu_n)$,
\begin{eqnarray*}
  (a_1 + a_2 + \cdots + a_n)^k  = \sum_{|{\mu}|=k}^{}\binom{k}{\mu}{\vec{a}}^{\mu},
\end{eqnarray*}
 where
 \begin{equation*}
  \binom{k}{\mu} = \frac{k!}{{\mu}!},\;\;\;\;{\vec{a}}^{\mu}=
{a_1}^{\mu_1} \times {a_2}^{\mu_2} \times \cdots \times
{a_n}^{\mu_n};\;\;\;  k \in \N.
\end{equation*}
In general, in HFT power series are considered as ordered by homogeneous Clifford holomorphic polynomials (cf. \cite{guehasp08, ma90b}) which in terms of several hypercomplex variables can be written analogously to the real case as
$$
f(\vec{z}) = \sum_{k=0}^{\infty} \left( \sum_{{|\mu|}=k}{}c_{\mu}
\vec{z}^{\mu} \right)\;\;\mbox{resp.}\;\; f(\vec{z}) = \sum_{k=0}^{\infty} \left( \sum_{{|\mu|}=k}{} \vec{z}^{\mu}c_{\mu} \right).
$$
Obviously, if the coefficients $c_{\mu}$ are real, we must not distinguish between left and right Clifford holomorphic series or left and right Clifford holomorphic polynomials. In fact, the generalized sequences of Appell polynomials we are dealing with in the last section, are such polynomials of the form \begin{equation}\label{polynGen}
{\bf P}_k(z_1,\cdots,z_n)=\sum_{{|\mu|}=k}{}c_{\mu}\vec{z}^{\mu}, \;c_{\mu} \in \mathbb{R}\; .
\end{equation}
Their representation in terms of several hypercomplex variables is not the only one possible. Also representations with respect to $x, \bar{x} \in \A_n$ and the scalar part $x_0$ and the vector part $\xv $ of $x \in \A_n$ separately reveal interesting properties and will be the subject of the next section.

\subsection{Generalized Appell sequences of homogeneous Clifford holomorphic polynomials}

The monomials $x^k,\;k\geq 0, $ are the prototype of an Appell sequence of the real variable $x \in \R$, \cite{app80}. The difficulty to generalize them in the HFT context has to do with the fact that only in the complex case ($n=1$) the power function $f(x)= x^n,\;x \in \A_n,$ belongs to the set of Clifford holomorphic functions since $\overline{\partial}x^n=\frac{1}{2}(1-n),\, n \in \N.$ To overcome this problem, the first approaches to generalized sequences of Appell polynomials in HFT \cite{fama06, fama07}, motivated also by geometric features, lead to sequences with particular relevance, for instance, for complete orthogonal polynomial systems and similar problems, cf. \cite{bogue10,lav12}.

We use the classical definition of sequences of Appell polynomials in \cite{app80} adapted to the hypercomplex case like it was done in \cite{fama07}.
\begin{definition}\label{defhypApp}
A sequence of \textit{homogeneous Clifford holomorphic polynomials } $\left(\mathcal{F}_k\right)_{k\geq0}$ of degree $k$ is called a \textit{generalized Appell sequence with respect to} $\partial$ if $\mathcal{F}_0(x)\equiv 1,$ $\mathcal{F}_k$ is of exact degree $k$ ($k\geq0$) and $\partial\mathcal{F}_k=k\,\mathcal{F}_{k-1}, \ k=1,2,\ldots$.
\end{definition}

One possibility to determinate $\left(\mathcal{F}_k\right)_{k\geq0}$ in HFT, like in complex function theory, is the expansion of a Clifford holomorphic function into its Taylor series through a generalized Cauchy integral formula and the expansion of the corresponding Cauchy kernel \cite{guehasp08}. Analogously, this expansion leads to a higher order geometric series. The generalized Cauchy kernel $\frac{1-\bar{x}}{\left|1-x\right|^{n+1}}, \; x \in \mathcal{A}_n$, is right and left monogenic for $|x|<1$. Taking into account that the variables $x$ and $\bar{x}$ are commuting variables because of $x\bar{x}=\bar{x}x=|x|^2 \in \mathbb{R},$ we obtain
\begin{align}\label{CauchyKernelExp}
\frac{1-\bar{x}}{\left|1-x\right|^{n+1}}=&\frac{1-\bar{x}}{\left[(1-x)(1-\bar{x})\right]^{\frac{n+1}{2}}}=
\frac{1}{(1-x)^{\frac{n+1}{2}}}\frac{1}{(1-\bar{x})^{\frac{n-1}{2}}}\nonumber\\
=&\sum_{k=0}^{+\infty}\left(\sum_{s=0}^{k}\frac{\left(\frac{n+1}{2}\right)_{k-s}}{(k-s)!}\frac{\left(\frac{n-1}{2}\right)_{s}}{s!}\right)x^{k-s} \bar{x}^s,
\end{align}
where $(.)_k$ stands for the Pochhammer symbol defined by $(a)_s=\frac{\Gamma(a+s)}{\Gamma(a)}$ or $(a)_s=a(a+1)(a+2)\ldots(a+s-1),\, (a)_0 =1, \;\;s\geq 0.$

Taking
\begin{equation}\label{tks}
T_{s}^{k}(n):= \binom{k}{s}\frac{(\frac{n+1}{2})_{k-s}(\frac{n-1}{2})_{s}}{(n)_{k}}
\end{equation}
and
\begin{equation}\label{pkn2}
\Pkn(x):= \sum_{s=0}^{k}T_{s}^{k}(n)\,x^{k-s}\,{\bar x}^s,
\end{equation}
we can write (\ref{CauchyKernelExp}) in the form\footnote{The reader recognizes in the Cauchy product which constitute the reduced form of the generalized Cauchy kernel in $\R^{n+1}$, the Chu-Vandermonde convolution identity $(a+b)_{k} = \sum_{s=0}^k\binom{k}{s}(a)_{k-s}(b)_{s}$ since $(a+b)_{k}=(n)_{k.}$ }
\begin{align*}
\frac{1-\bar{x}}{\left|1-x\right|^{n+1}}=\sum_{k=0}^{+\infty}\frac{(n)_{k}}{k!}\,\Pkn(x).
\end{align*}

 The next step shows immediately the relation of (\ref{tks}) to Vietoris' number sequence (\ref{equ1a}) and its generalization. Considering $x_0=0$ in (\ref{pkn2}), it follows
\begin{align*}
\Pkn(\xv) =& \sum_{s=0}^{k}(-1)^s \,T_{s}^{k}(n)\,{\xv}^k
= c_k(n)\,{\xv}^k,
\end{align*}
where
\begin{equation}\label{genCoeff}
c_k(n):=\sum_{s=0}^{k}(-1)^s \,T_{s}^{k}(n).
\end{equation}
Moreover, in \cite{fama12a}, it was proved that the sequence $\left(c_k(n)\right)_{k\geq0}$ ($n \in \mathbb{N}$) defined by (\ref{genCoeff}) has the property
\begin{equation*}
c_{0}(n)=1, \,c_{2m}(n)=c_{2m-1}(n),\; \text{for}\; m\geq 1\;\text{and}\;n \in \mathbb{N},
\end{equation*}
and, therefore, constitutes a \textit{generalized Vietoris' number sequence}. In fact, for $n=2$, we get exactly the Vietoris' sequence (\ref{equ1a}).

Observing that
$x_0= (x+\bar{x})/2\;\;\mbox{and}\;\; \xv=(x-\bar{x})/2$, another representation of (\ref{pkn2}) can be derived.
This leads to
\begin{equation*}\label{pkn}
\P_{k}^n(x)={\sum_{s=0}^k\binom{k}{s}c_{s}(n)\,x_{0}^{k-s}\,\xv^{s}},
\end{equation*}
considered in detail in \cite{cafama12}, where the coefficients $c_s(n)$ ($s=0,\ldots,k$) are explicitly written as
\begin{equation}\label{ck}
c_s(n)=
\begin{cases}
\frac{s!!(n-2)!!}{(n+s-1)!!},&\text{\ if } s \text{\ is odd}\\[1ex]
c_{s-1}(n),&\hbox{\ if } s \text{\ is even}
\end{cases}.
\end{equation}
A different representation of the coefficients \eqref{genCoeff} (or \eqref{ck}) in terms of the Pochhammer symbol was introduced in \cite{cafama15}:
\begin{equation} \label{coeffPocha}
c_{2s}\left(n\right)=c_{2s-1}\left(n\right)=\frac{\left(\tfrac{1}{2}\right)_s}{ \left(\frac n2\right)_s}.
\end{equation}
From this representation it is even more clear that the number sequence $\left(c_k(n)\right)_{k\geq0}$ ($n \in \mathbb{N}$), whose elements satisfy (\ref{coeffPocha}) is a generalization of Vietoris' number sequence (\ref{equ1a}). Indeed, for $n=2$, we get easily
\begin{equation*}
c_{2k}\left(2\right)=c_{2k-1}\left(2\right)=\frac{\left(\tfrac{1}{2}\right)_k}{ k!}=
\frac{1}{2^{2k}}\binom{2k}{k}, \;\; k\geq 0.
\end{equation*}

\section{Generating functions and properties}

This section is devoted to the study of generating functions for generalized Vietoris' number sequences $\left(c_k(n)\right)_{k\geq0}$ ($n \in \mathbb{N}$). Despite the natural appearance of those real numbers in HFT, the methods to obtain their generating functions rely on simple real analysis. We start by observing that the representation (\ref{coeffPocha}) of its elements in terms of quotients of numbers represented by the Pochhammer symbol suggests the use of the well known Gauss' hypergeometric function. We recall that the Gauss' hypergeometric function is defined by
\begin{equation}\label{GeralHyperF}
\displaystyle {}_{2}F_{1}(a,b;c;z)= \sum_{k=0}^{+\infty} \frac{(a)_k\,(b)_k}{(c)_k} \frac{z^k}{k!}, \;\;|z|<1,
\end{equation}
where $a,b \in \mathbb{C}$, $c \in \mathbb{C}\setminus \left\{\mathbb{Z^{-}}\cup \{0\}\right\}$. For $z \in \mathbb{C}$ outside the circle of convergence, the hypergeometric function ${}_{2}F_{1}$ is defined by analytic continuation.

On the unit disc $|z|=1$, the series converges absolutely when $Re(c-a-b)>0$ and
\begin{equation}\label{HypFctSumm}
{}_{2}F_{1}(a,b;c;1)=\frac{\Gamma(c)\Gamma(c-a-b)}{\Gamma(c-a)\Gamma(c-b)},
\end{equation}
it converges conditionally for $-1<Re(c-a-b)\leq 0$ and if  $Re(c-a-b)\leq -1$ it diverges. Moreover, we use in the sequel also the following properties of (\ref{GeralHyperF}):
\begin{equation}\label{HypFctProp1}
{}_{2}F_{1}(a,b;b;z)=\left(1-z\right)^{-a}
\end{equation}
and
\begin{equation}\label{HypFctProp2}
{}_{2}F_{1}(a,a+\frac12;1+2a;z)=\left(\frac12+\frac12\sqrt{1-z}\right)^{-2a}.
\end{equation}

We are now ready to formulate the main result of this section.
\begin{teo}
Let $G(.,n)$ be the following real-valued function depending on a parameter $n \in \mathbb{N}$:
\begin{equation}\label{GenerFunct}
G(t;n)=
\begin{cases}
\frac{1}{t}\left[(1+t)\,{}_{2}F_{1}(\frac12,1;\frac n2;t^2)-1\right],&\text{\ if } t\in ]-1,0[\cup]0,1[\\[1ex]
1,&\hbox{\ if } t=0.
\end{cases}
\end{equation}
Then, for any fixed $n \in \mathbb{N}$, $G(.,n)$ is a one-parameter generation function of the sequence $\left(c_k(n)\right)_{k\geq0}$.
\end{teo}
\begin{proof}
For each fixed $n \in \mathbb{N}$, consider the sequence $\left(c_k(n)\right)_{k\geq0}$ whose terms are defined by \eqref{coeffPocha}. The one-parameter generating function for this sequence can be written as the following formal power series in the real variable $t$
\begin{align}\label{generFct1}
G(t;n)=
\sum_{k=0}^{+\infty} c_{k}\left(n\right) \,t^{k}=&\sum_{s=0}^{+\infty} c_{2s}\left(n\right) \,t^{2s}+\sum_{s=1}^{+\infty} c_{2s-1}\left(n\right) \,t^{2s-1}\nonumber\\
=&\frac{1+t}{t}
\sum_{s=0}^{+\infty} c_{2s}\left(n\right) \,t^{2s}-\frac1t,
\end{align}
for $t\neq0$, because $c_{2s}\left(n\right)=c_{2s-1}\left(n\right)$ ($s=1,2,\ldots$).

Taking $a=\frac12$, $b=1$, $c=\frac n2$ ($n \in \mathbb{N}$) and $z=t^2$ ($t \in \left]-1,1\right[\setminus\{0\}$) in \eqref{GeralHyperF}, we obtain for \eqref{generFct1},
\begin{equation*}
G(t;n)=\frac{1+t}{t}
{}_{2}F_{1}(\frac12,1;\frac n2;t^2)-\frac1t.
\end{equation*}

Since ${}_{2}F_{1}(\frac12,1;\frac n2;0)=1$, the function (\ref{GenerFunct}) is well-defined for $|t|<1$.
\end{proof}

It is clear that we can obtain a closed formula for the generating function $G(.;n)$ of the sequence $\left(c_k(n)\right)_{k\geq0}$ ($n \in \mathbb{N}$) as long as it is known a closed formula for the corresponding hypergeometric series. As examples we list some cases where such closed formulae are well known and, consequently, closed formulae for $G(.;n)$ can be easily obtained.

\subsection{Examples}
\begin{enumerate}
\item $n=1$; in this case, $c_k(1)=1$ ($k\geq0$) and the corresponding generating function is given by
\begin{align*}
G(t;1)=\frac{1}{t}\left[(1+t)\,{}_{2}F_{1}(\frac12,1;\frac 12;t^2)-1\right]=\frac{1}{1-t},
\end{align*}
because ${}_{2}F_{1}(\frac12,1;\frac 12;t^2)$ reduces to the geometric function.

Notice that the case $n=1$ correspond to the complex case in HFT.

\item $n=2$; in this case, $c_{2k}(2)=c_{2k-1}(2)=\frac{\left(\frac12\right)_k}{k!}$($k\geq0$) and by using \eqref{HypFctProp1} the corresponding generating function is obtained as
\begin{align*}
G(t;2)=\frac{1}{t}\left[(1+t)\,{}_{2}F_{1}(\frac12,1;1;t^2)-1\right]=\frac{\sqrt{1+t}-\sqrt{1-t}}{t \sqrt{1-t}}.
\end{align*}

It is worth to notice that this case corresponds to Vietoris' number sequence (\ref{equ1a}). Moreover, we remark that this real-valued function also generates the sequence (\ref{pattern}), defined in terms of quaternion units.

\item $n=3$; $c_{2k}(3)=c_{2k-1}(3)=\frac{\left(\frac12\right)_k}{\left(\frac32\right)_k}$ ($k\geq0$) and the corresponding generating function is given by
\begin{align*}
G(t;3)=&\frac{1}{t}\left[(1+t)\,{}_{2}F_{1}(\frac12,1;\frac32;t^2)-1\right]=\frac 1t\left(\frac{t+1}{t}\ln{\sqrt{\frac{1+t}{1-t}}}-1\right).
\end{align*}
Here we computed a closed formula  for the function ${}_{2}F_{1}(\frac12,1;\frac32;t^2)$ by observing that $\frac{\left(\frac12\right)_k}{\left(\frac32\right)_k}=\frac{1}{2k+1}$ and using integration.

\item $n=4$; $c_{2k}(4)=c_{2k-1}(4)=\frac{\left(\frac12\right)_k}{(k+1)!}$ ($k\geq0$) and by using \eqref{HypFctProp2} the corresponding generating function leads to
\begin{align*}
G(t;4)=\frac{1}{t}\left[(1+t)\,{}_{2}F_{1}(\frac12,1;2;t^2)-1\right]=\frac{2t+1-\sqrt{1-t^2}}{t (1+\sqrt{1-t^2})}.
\end{align*}
\end{enumerate}

\subsection{Series involving generalized Vietoris' number sequences}
Using the one-parameter generating function $G(.;n)$ ($n \in \mathbb{N}$), we can now study the convergence of some non-trivial series that involve generalized Vietoris' number sequences.

As a consequence of the properties of the Gauss' hypergeometric function, $G(.;n)$ is well-defined for $t=\pm1$  if  $n>3$.

For $t=1$, $G(1;n)$ corresponds to the series whose general term is $c_k(n)$ ($k \geq 0$ and $n \in \mathbb{N}$). Using (\ref{HypFctSumm}), its sum is equal to
\begin{align*}
\sum_{k=0}^{+\infty} c_{k}\left(n\right)=\frac{n-1}{n-3},\;\text{for}\, n>3.
\end{align*}

The case $n=1$ leads clearly to a divergent series as consequence of the properties of the Gauss' hypergeometric function. For $n=2$, the lower bound $\binom{2k}{k}\geq\frac{2^{2k}}{2k+1}$ ensures that the corresponding series is divergent. The case $n=3$ leads to the series of reciprocal odd numbers, which is clearly divergent.

On the other hand $G(-1;n)$ is also well-defined for $n=2$ and $n=3$ because it leads to the series
\begin{equation*}
\sum_{k=0}^{+\infty} (-1)^k c_{k}\left(n\right)
\end{equation*}
that is convergent by Leibniz' test, for $n>1$. Indeed, taking into account formula (\ref{coeffPocha}), it is clear that the sequence $\left(c_{k}(n)\right)_{k\geq 0}$ is decreasing to zero, except for $n=1$, where the corresponding series $\sum_{k=0}^{+\infty} (-1)^k $ is divergent.

\section{Representation of the generalized Vietoris' number sequences by Clifford-algebra generators}
Now we are able to deduce the formula that generalizes (\ref{pattern}) and expresses the generalized Vietoris' number sequence (\ref{coeffPocha}) for an arbitrary value $n$ only by the generators $e_1, e_2, \dots e_n$ of the Clifford algebra $\Cl$. Therefore, we consider the polynomials $\Pkn$ in the form (\ref{pkn2}), for $x=1$. This yields to
\begin{equation}\label{tsksum}
   \Pkn(1) = \sum_{s=0}^kT_s^k(n)=\sum_{s=0}^k\frac{k!}{(n)_{k}}\frac{(\frac{n+1}{2})_{k-s}(\frac{n-1}{2})_{s}}{(k-s)!
s!}=1,
\end{equation}
where (\ref{tsksum}) follows from Chu-Vandermonde's identity for the Pochhammer symbol to the sum of all of $T_s^k(n)$ for fixed $k$ and $n$ (see footnote 3).

\begin{teo}
For each fixed $n \in \mathbb{N}$,  the elements of the generalized Vietoris' number sequence $\left(c_k(n)\right)_{k\geq0}$ admit the representation
\begin{equation} \label{cknn}
c_{k}(n)=(-1)^k \left[\sum_{|\nu|=k}
\binom{k}{\nu}\left(e_1^{\nu_1}\times e_2^{\nu_2}\times \cdots
\times e_n^{\nu_n}\right)^2\right]^{-1}.
\end{equation}
\end{teo}
\begin{proof}
Using the general form (\ref{polynGen}) of a homogenous Clifford holomorphic polynomial of degree $k$, the exact expression of the Appell polynomial (\ref{pkn2}) is (cf. \cite{mafa10})
\begin{equation*} \label{pks}
\Pkn(z_1,\cdots,z_n)= c_k(n)
\sum_{|\nu|=k} \binom{k}{\nu} z_1^{\nu_1}\times\cdots\times z_n^{\nu_n}\cdot
 e_1^{\nu_1}\times\cdots\times e_n^{\nu_n}.
\end{equation*}
Setting $x_0 =1$ and $x_k=0$ in $z_k=x_k-x_0e_k$ ($k=1,\ldots,n$) from (\ref{tsksum}) it follows
\begin{equation*} \label{pks1}
1= \Pkn(-e_1,-e_2,\ldots,-e_n)= c_k(n)
\sum_{|\nu|=k} \binom{k}{\nu} (e_1^{\nu_1}\times\cdots\times e_n^{\nu_n})^2
\end{equation*}
and we obtain (\ref{cknn}).
\end{proof}

In \cite{mafa10} the reader can also find some relations of the generalized Vietoris' sequence (\ref{coeffPocha}) with special values of Bessel functions and Legendre polynomials.

\section{Conclusion}
Together with \cite{rusa04} the results show that Vietoris' sequence of rational numbers combines seemingly disperse subjects in Real, Complex and Hypercomplex Analysis. They also show that a non-standard application of Clifford algebra tools was able to reveal these new insights in objects of combinatorial nature.

\section*{Acknowledgements}

 This work was supported by Portuguese funds through the CIDMA-Center for Research and Development in Mathematics and Applications, and the Portuguese Foundation for Science and Technology (``FCT-Funda\c c\~ao para a Ci\^encia e Tecnologia"), within project UID/MAT/04106/2013.


\begin{thebibliography}{99}

\bibitem{app80}
P.\, Appell, Sur une classe de polynomes, { Ann. Sci. \`Ecole Norm. Sup.}
{9} (2) (1880) 119-144.

\bibitem{ask94}
R. Askey, Orthogonal polynomials and special functions, Society for Industrial and
Applied Mathematics, Philadelphia, 2nd ed. 1994.

\bibitem{asst74}
R. Askey, J. Steinig, Some positive trigonometric sums, Transactions AMS
{187} (1) (1974) 295-307.

\bibitem{bogue10}
S.~Bock and K.~G{\"u}rlebeck, On a Generalized Appell System and Monogenic Power Series, Math. Methods Appl. Sci. {33} (4) (2010) 394--411.

\bibitem{cafama12}
I.~Ca\c{c}\~ao, M.~I.~Falc\~ao,  H.~R.~Malonek, Matrix representations
of a basic polynomial sequence in arbitrary dimension,  Comput. Methods Funct.
Theory 12 (2) (2012)  371-391.

\bibitem{cafama15}
 I. Ca\c{c}\~{a}o, M. I. Falc\~{a}o and H. R. Malonek, Three-Term Recurrence Relations for Systems of Clifford Algebra-Valued Orthogonal Polynomials, 15p.,
doi:10.1007/s00006-015-0596-z

\bibitem{desosu92}
R. Delanghe, F. Sommen, V. Sou\v{c}ek, Clifford Algebra and Spinor-Valued Functions. A function theory for the Dirac operator,
Mathematics and its Applications (Dordrecht) 53  Kluwer Academic Publishers, 1992.

\bibitem{fama06}
M.\, I.\, Falc\~{a}o, J.\, F.\, Cruz, H.\, R.\,  Malonek, Remarks on the generation of monogenic  functions.  In: K.\,G\"{u}rlebeck, C.\, K\"onke (Eds.), $17^{th}$ Inter. Conf. on the Appl. of Computer Science and Mathematics in Architecture and Civil Engineering, Weimar, 2006, pp. 12-14.

\bibitem{fama07}
 M.\, I.\, Falc\~{a}o,  H.\, R.\,  Malonek,
 Generalized exponentials through Appell sets in $\R^{n+1}$ and Bessel functions. In: T.\,E.\,Simos, G.\, Psihoyios, C.\, Tsitouras (Eds.), AIP Conference Proceedings 936, 2007, pp. 738-741.

 \bibitem{fama12a}
  M.\, I.\, Falc\~{a}o,  H.\, R.\,  Malonek, A note on a one-parameter family of non-symmetric number triangles,
{ Opuscula Mathematica} {32} (4) (2012), 661-673.

\bibitem{guehasp08}
K.~G{\"u}rlebeck, K.~Habetha, and W.~Spr{\"o}{\ss}ig, Holomorphic Functions in the Plane and {$n$}-Dimensional Space, Translated from the 2006 German original. Birkh\"auser Verlag, Basel, 2008.
%
\bibitem{guema99}
K.\,G\"{u}rlebeck, H.\, R.\,  Malonek, A hypercomplex derivative of monogenic functions in $\R^{m+1}$ and its applications, { Complex Variables}
{39}  (1999) 199-228.
%
\bibitem{lav12}
R.~L{\'a}vi{\v{c}}ka, Complete Orthogonal Appell Systems for Spherical Monogenics, Complex Anal. Oper. Theory, {6} (2012) 477--489.

\bibitem{ma90b}
H.\, R.\,  Malonek, Power series representation for monogenic functions in ${\R}^{n+1}$ based on a permutational product, { Complex Variables, Theory Appl.} {15} (1990) 181-191.

\bibitem{mafa10}
H.\, R.\,  Malonek, M.\, I.\, Falc\~{a}o, On Special Functions in the Context of Clifford Analysis. In: T.\,E.\,Simos, G.\, Psihoyios, C.\, Tsitouras (Eds.), AIP Conference Proceedings 1281, 2010, pp. 1492-1495.

\bibitem{rusa04}
St. Ruscheweyh, L. Salinas, Stable functions and Vietoris'
theorem, J. Math. Anal. Appl. {291} (2004) 596-604.

\bibitem{viet58}
L. Vietoris, \"{U}ber das Vorzeichen gewisser trigonometrischer
Summen, Sitzungsber. \"{O}sterr. Akad. Wiss 167, (1958) 125-135.
\end{thebibliography}
\end{document}